\documentclass{amsart}
\usepackage[english]{babel}
\usepackage[utf8]{inputenc}
\usepackage{amsmath, amssymb,epic,graphicx,mathrsfs,enumerate}
\usepackage[all]{xy}
\usepackage{color}
\usepackage{comment}
\usepackage{enumitem}
\usepackage[]{frontespizio}
\usepackage{hyperref}

\usepackage{amsmath}
\usepackage{amsthm}
\usepackage{amssymb}
\usepackage{amsfonts}
\usepackage{latexsym}
\usepackage{epsfig}

\DeclareMathOperator{\aut}{Aut} 
\DeclareMathOperator{\soc}{soc}

\newcommand{\psl}{\mathrm{PSL}}
\newcommand{\psu}{\mathrm{PSU}}
\newcommand{\g}{\mathrm{G}}

\newtheorem{theorem}{Theorem}[section]
\newtheorem{question}[theorem]{Question}
\newtheorem{proposition}[theorem]{Proposition}
\newtheorem{corollary}[theorem]{Corollary}
\newtheorem{lemma}[theorem]{Lemma}
\theoremstyle{definition}
\newtheorem{definition}[theorem]{Definition}
\newtheorem{remark}[theorem]{Remark}
\numberwithin{equation}{section}

\renewcommand{\footnote}{\endnote}
\newcommand{\ignore}[1]{}\makeglossary

\begin{document}
	\bibliographystyle{amsplain}

\title[On finite groups whose power graph is claw-free]{On finite groups whose power graph is claw-free}
\author{Pallabi Manna}
\address{Pallabi Manna, Harish Chandra Research Institute, Prayagraj - 211019, India}
\email{mannapallabimath001@gmail.com}  
\author{Santanu Mandal}
\address{Santanu Mandal, School of Computing Science and Engineering, VIT Bhopal University, Bhopal-466114, India}
\email{santanu.vumath@gmail.com}
\author{Andrea Lucchini}
\address{A. Lucchini, Universit\`a di Padova, Dipartimento di Matematica ``Tullio Levi-Civita'', Via Trieste 63, 35121 Padova, Italy}
\email{lucchini@math.unipd.it}

\begin{abstract}A graph is called claw-free if it contains no induced subgraph isomorphic to the complete bipartite graph $K_{1, 3}$. The undirected power graph 
	 of a group $G$ has vertices the elements of $G$, with an edge between $g_1$ and $g_2$ if one of the two cyclic subgroups $\langle g_1\rangle, \langle g_2\rangle$ is contained in the other. It is denoted by
$P(G)$. The reduced power graph, denoted by $P^*(G),$ is the subgraph of $P(G)$ induced by the non-identity elements. The main purpose of this paper is to explore the finite groups whose reduced power graph is claw-free. In particular we prove that if $P^*(G)$ is claw-free, then either $G$ is solvable or $G$ is an almost simple group. In the second case the socle of $G$ is isomorphic to $\psl(2,q)$ for suitable choices of $q$.
Finally we prove that if $P^*(G)$ is claw-free, then the order of $G$ is divisible by at most 5 different primes.
\end{abstract}


\maketitle

\section{Introduction}
There are numerous significant and well-known families of graphs that are likewise claw-free. This family includes line graphs, complement of triangle-free graphs, the graph of the icosahedron, the comparability graphs, and so on. The initial incentive to investigate the characteristics of claw-free graphs seems to have emerged from Beineke's characterization of line graphs in \cite{Beineke}. However, it was primarily in the late 1970s and early 1980s that the graph theory community began to pay attention to the class of claw-free graphs. In this time frame, the matching number of these graphs was noticed, and initial findings on Hamiltonian properties were validated. Still, perhaps more significant were the findings that the independence number can be found using a polynomial method and that Berge's Perfect Graph Conjecture holds true for the class of graphs that are free of claws.

In algebraic graph theory we enjoy fantastic applications of abstract algebra in graph theory. In this context, numerous graph structures to algebraic structures, such as groups, semigroups, rings, etc., were defined by various academics. This type of research aids in the characterization of the resultant graphs, the exploration of the algebraic structure from the isomorphic graph, and the recognition of a relationship between the related graphs and algebraic structures. Graph defined on groups has lot of collections such as the generating graph, enhanced power graph, commuting graph, soluble graph, and so on. From ancient to the latest collection, interested readers may follow the articles \cite{acns, Nath, Dey,Burness, pjc, Lucchini}. The power graph is one of them. The concept of power graph was initially presented by Kelarev and Quinn in 2002 (\cite{Kelarev}) as a directed graph. 
\begin{definition}
	The directed power graph of a group $G$, denoted by $\overrightarrow{P}(G)$, is a graph whose vertices are the elements of the group $G$, and there is an arc $u\rightarrow v$ ($u\neq v$) if and only if $v=u^{i}$, for some positive integer $i$.
\end{definition}
Later in 2009, Chakrabarty et al. defined the undirected power graph \cite{Chakrabarty}.
\begin{definition}
	The undirected power graph (or simply power graph) of a group $G$, denoted by $P(G)$, is a graph whose vertices are the elements of the group $G$, in which two vertices $u$ and $v$ are connected by an edge  if and only if either $u=v^i$ or $v=u^j$ for some positive integers $i$, $j$. 
\end{definition}
The reduced power graph, denoted by $P^*(G),$ is the graph $P(G)\setminus \{e\}$, where $e$ is the identity of the group. In a variety of literature sources, power graphs have  investigated in several ways. We refer to two survey publications \cite{Abawajy, Kumar} for intriguing results regarding power graphs.

In terms of forbidden induced subgraphs, several important classes of graphs can be specified. Among other graph types, these classes include triangle-free graphs, claw-free graphs, chordal graphs, cograph, split, and threshold graphs. This topic was introduced for the different forbidden induced subgraphs of power graphs in an earlier study \cite{Doostabadi, cmm}. In \cite{Mehatari}, the authors answered the question: \lq\lq Which power graphs are cograph?\rq\rq. Bera discussed the line graph characterization of the power graph in \cite{Bera}. 

In the power graph, the identity element is adjacent to the other vertices of the graph. Therefore the condition \lq$P(G)$ is claw-free\rq \ can occur only if $G$ is either a cyclic group of prime power order or a cyclic group of the form $C_{p^{m}q}$, where $p, q$ are distinct primes and $m\geq 1$. However, the problem gets more intriguing if we remove the identity from the power graph. 

This paper addresses the problem: \lq \lq Identify the groups whose reduced power graph is claw-free\rq\rq. 

In Section \ref{prelim} we introduce some preliminary considerations which allow, in particular, to obtain some information on finite nilpotent groups whose reduced power graph is claw-free. We prove that {{\sl{if $G$ is a $p$-group
with $p$ odd and $P^*(G)$ is claw-free, then either $G$ is cyclic or $\exp(G)=p$}}. The situation is more complicate when $p=2:$ {\sl{if $G$ is a non-cyclic 2-group and $P^*(G)$ is claw-free, then either $G$ is a dihedral group or $\exp(G)\leq 4$}}. Moreover {\sl{a nilpotent group $G$ whose reduce power-graph is claw-free is either a $p$-group or a cyclic group whose order is divisible by at most 2 different primes}}.
Finally we prove that if $G$ is not nilpotent and $P^*(G)$ is claw-free, then either $G$ is indecomposable or there exist two distinct primes $p$ and $q$ such that
$G=H\times K$ where $H\cong C_p$ and $K$ is the semidirect product of a cyclic $q$-group $Q$  and
a cyclic $p$-subgroup $P$ of $\aut(Q).$

Our main result in the solvable case, proved in Section \ref{solv}, is the following: 
\begin{theorem}\label{fourprime}
	Let $G$ be a finite solvable group. If $P^*(G)$ is claw-free, then $\pi(G)\leq 4,$ denoting
 by $\pi(G)$ the number of distinct prime divisors of the order of $G.$
\end{theorem}

The previous result is best possible. For example let $C=\langle x_1, x_2\rangle,$ with $|x_1|=31,$ and $|x_2|=61;$ then $\langle x_1\rangle$ has an automorphism 
$\alpha_1$ of order 15 and $\langle x_2\rangle$ has an automorphism 
$\alpha_2$ of order 15. Let $\langle y\rangle \cong {C_{15}}.$ Then we may define a semidirect product $G=C\rtimes \langle y\rangle\cong C_{31\cdot 61}\rtimes C_{15},$ in which $(x_1x_2)^y=x_1^{\alpha_1}x_2^{\alpha_2}.$ It follows from Proposition \ref{squarefree} that  $P^*(G)$ is claw-free.

\

In Section \ref{unso} we investigate the structure of finite groups $G$ with claw-free reduced power graph, when the solvability assumption is removed. Our main result is:

\begin{theorem}\label{unsolva}
	Let $G$ be a finite group. If $G$ is not solvable and $P^*(G)$ is claw-free, then $G$ is an almost simple group.
\end{theorem}

Section \ref{simgr} contains a complete classification of finite simple groups whose reduced power graph is claw-free. 
It turn out that if $G$ is a finite non-abelian simple group and $P^*(G)$ is claw free, then $G\cong \psl(2,q)$. More precisely:

\begin{theorem}\label{simplenoclaw}
	Let $G$ be a finite non-abelian simple groups. Then
	$P^*(G)$ is claw-free if and only if
	$G\cong \psl(2,q)$ and each of the two numbers $\frac{q\pm 1}{(2,q-1)}$ is either a prime-power or the product of a prime and a prime-power.
\end{theorem}


Combining Theorems \ref{unsolva} and \ref{simplenoclaw} we finally  prove the following:

\begin{theorem}\label{fiveprime}
	Let $G$ be a finite group. If $P^*(G)$ is claw-free, then $\pi(G)\leq 5.$ 
\end{theorem}

The bound given in the previous statement is best possible.
For example
the order of $\psl(2,64)$ is divisible by the five primes
2, 3, 5, 7, 13, and $P^*(\psl(2,64))$ is claw-free (see Theorem \ref{psl2}).

\

This paper uses standard notations. To learn the fundamental concepts and symbols of group theory that are not covered here, the reader can consult any of \cite{isa, Robinson}.


\section{Some preliminary results}\label{prelim}
Assume that $P^*(G)$ contains a claw made by $3$ pendent vertices $a_1, a_2, a_3$ and a central vertex $b$.
For $1\leq i\leq 3,$ the directed graph $\overrightarrow{P}(G)$ contains at least one of the two
edges $a_i\to b$ and $b\to a_i.$ Moreover it cannot be
that $\overrightarrow{P}(G)$ contains $a_{i_1}\to b$ and
$b\to a_{i_2},$ for $a_{i_1}\neq a_{i_2}$, otherwise it would also contain the edge $a_{i_1}\to a_{i_2}.$ It follows that there are two possible types of claws in $P^*(G):$
\begin{enumerate}
	\item the subgraph of $\overrightarrow{P}(G)$ induced by the four vertices $a_1,a_2,a_3$ and $b$ consists of the three edges: $b\to a_1, b\to a_2, b\to a_3.$ In this case we will say that $a_1,a_2,a_3,b$ induces a claw of first type and we will use the symbol $(b \to a_1,a_2,a_3).$
	\item  the subgraph of $\overrightarrow{P}(G)$ induced by the four vertices $a_1,a_2,a_3$ and $b$ consists of the three edges: $a_1\to b, a_2\to b, a_3\to b.$ In this case we will say that $a_1,a_2,a_3,b$ induces  a claw of second type and we will use the symbol $(a_1,a_2,a_3 \to b).$
\end{enumerate}

Notice that if $(b\to a_1,a_2,a_3)$ is a claw of first type, then $\langle a_1\rangle, \langle a_2\rangle, \langle a_3\rangle$ are incomparable subgroups of the cyclic group $\langle b\rangle$: this implies that $|b|$ cannot be a prime power, since the subgroup lattice of a cyclic $p$-group is a chain. In particular if every non-trivial element of $G$ has prime-power order, then $P^*(G)$ does not contain claws of first type. Recall that a finite group is called an EPPO group if every non-identity element is of prime-power order.  EPPO groups were first introduced in 1950 by Higman \cite{Higman}; in 1960 Suzuki classified the simple EPPO groups \cite{suzuki2} and a complete classification can be found in \cite{hei}. 
Assume that $G$ is an EPPO group. Since $P^*(G)$ cannot contain claws of the first type, we deduce that $P^*(G)$ is claw-free if and only if $G$ does not contain three distinct cyclic subgroups with non-trivial intersection. In particular $P^*(G)$ is claw-free if all the non-trivial elements of $G$ have prime order (this occurs if and only if $G$ has exponent $p$, or $G$ is a Frobenius group of order $p^aq$ and the Sylow $p$-subgroup of $G$ has exponent $p$, or $G=A_5$ \cite{allp}).  

\begin{remark}\label{manyprimes}Let $g$ be a non-trivial element of a finite group $G.$
	\begin{enumerate}
		\item If $g$ has ordine $p_1p_2p_3$ where $p_1, p_2, p_3$ are distinct primes, then $(g\to g^{p_1},g^{p_2},g^{p_3})$ is a claw of first type in $P^*(G).$
		\item If $g$ has ordine $p_1^2p_2^2$ where $p_1, p_2$ are distinct primes, then $(g\to g^{p_1^2},g^{p_2^2},g^{p_1p_2})$ is a claw of first type in $P^*(G).$
	\end{enumerate}
\end{remark}

\begin{remark}
	If $G=\langle g \rangle$ is a cyclic group of order $p^aq,$ where $p$ and $q$ are distinct primes, then $P^*(G)$ is claw-free.
\end{remark}
\begin{proof}
Assume that there exists a claw in $P^*(G)$ with pendants $a_1, a_2, a_3.$ Since the subgroup lattice of $\langle g^q\rangle$ is a chain, at most one of the three elements $a_1,a_2,a_3$ can belong to $\langle g^q\rangle.$ So 
	we may assume $a_1=g^{p^{n_1}}$ and $a_2=g^{p^{n_2}}$ with $n_1\leq n_2$. But in this case $P^*(G)$ contains the edge $a_1\to a_2.$
\end{proof}

\begin{remark}\label{p2p}
	Let $p$ be an odd prime and $G=C_{p^2}\times C_p$ or let $p=2$ and $G=C_4\times C_2\times C_2.$ Then
	$P^*(G)$ contains a claw of the second type.
\end{remark}
\begin{proof}
	The group $G$ contains an element $g$ of order $p^2$ and three pairwise distinct subgroups $\langle a_1\rangle, \langle a_2\rangle, \langle a_3\rangle$ of order $p$, not contained in $\langle g\rangle$. It can be easily seen that
	$(a_1g,a_2g,a_3g\to g^p)$ is a claw in $P^*(G).$
\end{proof}
Notice that the first part of the previous statement is not true when $p=2.$ Indeed if $G=C_4\times C_4$, then every element of order 2 in $G$ is contained in exactly two cyclic subgroups of order 4 and therefore $P^*(G)$  is a claw-free graph. As a consequence,
$P^*(C_4\times C_2)$ is also claw-free. However, when $p=2,$ the following  weaker version of Remark \ref{p2p} holds.
\begin{remark}
	Let $G=C_{2}\times C_8.$ Then
	$P^*(G)$ contains  a claw of the second type.
\end{remark}
\begin{proof}If $a$ in an element of $G$ of order 8 and $b$ is an element of order 2 with $b\neq a^4,$ then $(a,ab,a^2b\to a^4)$ is a claw in $P^*(G).$
\end{proof}

\begin{corollary}
	\label{th_claw_abelian}
	Let $G$ be an abelian $p$-group. If $G$ is cyclic or has exponent $p,$ then $P^*(G)$ is claw-free. Moreover, if $G$   is neither  cyclic nor elementary abelian, then $P^*(G)$ is claw-free if and only if $p=2$ and either $G\cong C_{4}\times C_{2}$ or $G\cong C_{4}\times C_{4}$.
\end{corollary}

\begin{proposition}Let $p$ be an odd prime. If $G$ is a finite $p$-group, then $P^*(G)$ is claw-free if and only if either $G$ is cyclic or $G$ has exponent $p.$
\end{proposition}

\begin{proof}Assume that $G$ is not cyclic and contains an element $g$ of order $p^2$  and that $P^*(G)$ is claw-free. By the previous corollary, $G$ is non-abelian. 
First assume that $G$ is not of maximal class. By \cite[Exercise 5, \S 10]{berk}, there exists a non-cyclic abelian subgroup $H$ of $G$ that contains $g.$ By Remark \ref{p2p},
$P^*(H)$ contains a claw, so we get a contradiction. So we may assume that $G$ is of maximal class. But then $G$ contains a normal subgroup $N$ with $N\cong C_p\times C_p.$ Let $K=NC,$ with $C=\langle g\rangle$. It cannot be
$N\cap C=1$. Indeed in this case $(Z(K)\cap N)C$ is a non-cyclic abelian group of exponent $p^2$ and its reduce power graph contains a claw. We remains with the case when $K$ is a non-abelian $p$-group of exponent $p^2.$ But then $K \cong C_{p^2}\rtimes C_{p}=\langle a, b | a^{p^2}=b^{p}=1, bab^{-1}=a^{p+1}\rangle$ and $P^*(K)$ contains the claw $(a, ab, a^{2}b\to a^{p})$.
\end{proof}

It is possible to prove a weaker result when $p=2.$

\begin{proposition}\label{2groups}
	Let $G$ be a non-cyclic $2$-group. If $P^*(G)$ is claw-free then either $G$ has exponent at most 4 or $G$ is a dihedral group. 
\end{proposition}
\begin{proof}
Assume that $G$ is a non-cyclic $2$ group containing an element $x$ of order $8.$  By \cite[Proposition 10.30]{berk} either $G$ is a dihedral group (in which case it can be easily checked that $P^*(G)$ is claw-free) or one of the following holds:

\noindent 1) $G \cong {\rm{Q}}_{2^{n+1}}$ is a generalized quaternion group of order $2^{n+1},$ with $n\geq 3.$ In this case $G$ contains a cyclic normal subgroup $N$ of order $2^n$, a unique involution and all the elements in $G\setminus N$ have order 4. Thus $P^*(G)$ contains a claw.

\noindent 2) $G \cong {\rm{SD}}_{2^{n+1}}$ is a semidihedral group of order $2^{n+1},$ with $n\geq 3.$ In this case $G$ contains a subgroup $H$ isomorphic to  ${\rm{Q}}_{2^{n}}$, so $P^*(H)$ contains a claw.

\noindent 3) $G$ contains an elements $y$ of order $8$ such that $\langle x\rangle \cap \langle y \rangle$ has order 4. 
Let $K=\langle x,y\rangle.$ Then $K$ is a group of order 16, containing two different cyclic subgroups of order 8. There are two possibilities: $K\cong C_8\times C_2$ or $K\cong \langle a, b \mid a^8=1, b^2=1, a^b=a^5\rangle.$ In the first case we already know that $P^*(K)$ is not-claw free. In the second case, it can be easily checked that $(a,ab,a^2b\to a^4)$ is a claw in $P^*(K).$
\end{proof}

We need the following lemma to obtain more information on the 2-groups of exponent 4 whose reduced power graph is claw-free.

\begin{lemma}
Suppose that $G$ is a finite 2-group of exponent 4. If $P^*(G)$ is claw-free, then $G$ is $Q_8$-free.
\end{lemma}
\begin{proof}
	We prove by induction on $|G|$ that if $G$ has exponent 4
	and a section $X/Y$ of $G$ is isomorphic to the quaternion group $Q_8,$ then $P^*(G)$ contains a claw. 
	We may assume $Y\neq 1,$ otherwise $X\cong Q_8$ and $P^*(Y)$ contains a claw. Let $Z=\langle z\rangle$ be a subgroup of order 2 contained in $Y\cap Z(X).$ By induction $P^*(X/Z)$  contains a claw $(a_1Z,a_2Z,a_3Z\to bZ)$. For $1\leq i\leq 3,$ since $4=|a_iZ|\leq |a_i|\leq 4,$ we have $|a_i|=4$, which implies in particular $\langle a_i\rangle \cap Z=1.$ Moreover $b$ has order 2, $\langle b\rangle Z\cong C_2\times C_2$ and $\{a_1^2,a_2^2,a_3^2\}\ \subseteq \{b, bz\}.$ Thus we may assume $a_1^2=a_2^2$, but then $(a_1,a_2,a_1z\to a_1^2)$ is a claw in $P^*(G).$
	\end{proof}
	
Combining the previous lemma with the classification of the quaternion-free 2-group (see \cite{janko} and \cite{wil}), we obtain the following result:

\begin{proposition}\label{quatfree}
Let $G$ be a non-abelian 2-group of exponent 4. If $P^*(G)$ is claw-free, then $G=N\langle x\rangle$, where $N$ is an abelian normal subgroup of $G$. Moreover $|x|\leq 4$ and either $N$ is elementary abelian or $N\leq C_4\times C_4.$
\end{proposition}

\begin{proof}Assume that $G$ is a non-abelian 2-group of exponent 4 and that $P^*(G)$ is claw-free. By the previous lemma, $G$ is $Q_8$-free. By  \cite[Proposition 1.6]{janko}, if $G$ were modular, then the condition ${\rm{exp}}(G)=4$ would imply that $G$ is abelian.
So $G$ is non-modular and one of the three conditions $(a), (b), (c)$ described in \cite[Theorem 1.7]{janko} occurs. We may exclude case (c), since  it implies the existence of an element in $G$ of order at least 16. In cases (a) and (b) $G=N\langle x\rangle,$ where $N$ is a normal abelian subgroup.
Since $P^*(N)$ is claw-free, either $N\leq C_4\times C_4$ or $N$ is an elementary abelian 2-group.
 \end{proof}

 Although it is difficult to give a complete description of the finite $p$-groups whose reduced power graph is claw-free, it is easy to describe the nilpotent groups that have this property and are not $p$-groups. For this purpose, the following easy observation is needed.

\begin{remark}\label{ppq}If $G\cong C_p\times C_p\times C_q$, with $p$ and $q$ different primes, then $P^*(G)$ contains a claw of second type.
\end{remark}
\begin{proof}
	The group $G$ contains three different subgroups $\langle a_1\rangle, \langle a_2\rangle, \langle a_3\rangle$ of order $p$ and a subgroup $\langle b\rangle$ of order $q$. It suffices to notice that
	$(a_1b, a_2b, a_3b\to b)$ is a claw in $P^*(G).$
\end{proof}

\begin{proposition}\label{th_claw_nil}
	Suppose that $G$ is a finite nilpotent group and that $P^*(G)$ is claw-free. Then either $G$ is a $p$-group or $G$ is cyclic. In the second case $|G|=p^nq$ where $n\in \mathbb N$ and $p, q$ are distinct primes.
\end{proposition}
\begin{proof}
	If $\pi(G)\geq 3,$ then $\pi(\langle g\rangle)= 3$ for some $g\in G,$ and it follows from Remark \ref{manyprimes} (1) that $P^*(G)$ contains a claw. Hence we may assume that
	$G\cong P\times Q,$ where $P$ is a $p$-group and $Q$ is a $q$-group. Suppose that $Q$ is not cyclic. Then either $Q$ contains a subgroup isomorphic to $C_q\times C_q$ or $Q$ is a generalized quaternion group. It can be easily seen that if $Q$ is a generalized quaternion group, then $P^*(Q)$ contains a claw. On the other hand if $Q$ contains $C_q\times C_q$ then $G$ contains $C_q\times C_q\times C_p$ and $P^*(G)$ contains a claw by the previous remark. So if $G$ is claw-free then $Q$ (and $P$ for the same reason) is cyclic. Hence $G$ itself is cyclic and the conclusion follows from Remark \ref{manyprimes} (2).
\end{proof}


\begin{proposition}\label{centralizer}Let $G$ be a finite group, $g$ a non-trivial element of $G$ and $C=C_G(g).$ If $P^*(C)$ is claw-free then $\pi(C)\leq 2.$ Moreover if $\pi(C)=2,$ then one of the Sylow subgroups of $C$ is cyclic and normal.
\end{proposition}

\begin{proof}
It is not restrictive to assume $|g|=p,$ with $p$ a prime. Assume that $q$ is a different prime divisor of $|C|$ and let $Q$ be a Sylow $p$-sybgroup of $C.$ Since $\langle g\rangle Q$ is nilpotent, by Proposition \ref{th_claw_nil}, $Q$ is cyclic.
	Moreover $Q$ is normal in $C,$ otherwise $C$ would contain three different Sylow $q$-subgroups, say $\langle y_1\rangle,
\langle y_2\rangle, \langle y_3\rangle$ and $P^*(C)$ would contain the claw $(gy_1,gy_2,gy_3\to g).$ If $q_1$ and $q_2$ were two different prime divisors of $|C|,$ with $p\notin \{q_1,q_2\},$ then the $q_1$-Sylow subgroup and the $q_2$-Sylow subgroup would both be normal and $C$ would contain an element $y$ of order $q_1q_2.$ But then $gy$ would have order $pq_1q_2$ and $C$ would contain a claw by Remark \ref{manyprimes}.
\end{proof}

\begin{theorem}\label{inde}Assume that $G$ is a finite non-nilpotent group that can be decomposed as direct product of non-trivial factors. If
	$P^*(G)$ is claw free, then there exist two distinct primes $p$ and $q$ such that
	$G=H\times K$ where $H\cong C_p$ and $K$ is the semidirect product of a cyclic $q$-group $Q$  and
	a cyclic $p$-subgroup $P$ of $\aut(Q).$ Moreover either $P\cong C_p$ or $p=2$ and $P\cong C_4.$
\end{theorem}
\begin{proof}
	Suppose that $G=H\times K,$ with $H\neq 1$ and $K$ not nilpotent. Then, for a suitable prime divisor $p$ of $|K|,$ $K$ contains a Sylow $p$-subgroup, say $P$, which is not normal in $K$. 
	
	We claim that $H$ is a $p$-group. Indeed, assume that $H$ contains an element $\langle x\rangle$ of prime order $q,$ with $q\neq p.$ Then the reduced power graph $\langle x\rangle \times P$ is claw-free, and by Proposition \ref{th_claw_nil}, $P$ is cyclic. Since $P$ is not normal in $K$, $K$ contains three distict Sylow $p$-subgroups $P_1=\langle y_1\rangle, P_2=\langle y_2\rangle, P_3=\langle y_3\rangle.$ But then $P^*(G)$ contains the claw $(xy_1,xy_2,xy_3\to x).$ 
	
	Since $G$ is not nilpotent, there exists a prime $q\neq p$ dividing $|K|.$ Since $K$ centralizes the elements of $H$, it follows from the Proposition \ref{centralizer} and its proof that $q$ is the unique prime divisor of $|K|$ different from $p$ and that the Sylow $q$-subgroup $Q$ of $G$ is cyclic and normal in $C.$

Since $H\times Q$ is nilpotent and its reduced power graph is claw-free, $H$ is a cyclic $p$-group. Moreover we have proved  that $K=QP$ with $P$ a $p$-group and $Q$ a normal cyclic $q$-subgroup. We claim that $|H|=p.$ Otherwise $H$ contains an element $x$ of order $p^2$ and $K$ an element $y$ of order $p$ and an element $z$ of order $q$. In this case
	$(x,xy,x^pz\to x^p)$ is a claw in $P^*(G).$ We have also $C_P(Q)=1,$ otherwise $K$ contains an element of order $pq$ and $C_p\times C_p\times C_q$ is isomorphic to a subgroup of $G.$ Thus $P\leq \aut(Q)$ is cyclic. Finally, since $P^*(H\times P)$ is claw-free, either $P\cong C_p$ or $p=2$ and $P\cong C_4.$
\end{proof}

\section{Solvable groups}\label{solv}
The aim of this section is to prove Theorem \ref{fourprime}.
A key ingredient in the proof is the following property of Frobenius complements. 
\begin{lemma}\cite[Lemma 2.1]{morgan}
	\label{frocomp} 
	Suppose $X$ is a Frobenius complement. Then every Sylow subgroup of $X$ is cyclic
	or generalized quaternion. If $X$ has an odd order, then any two elements of prime order commute, and $X$ is metacyclic.
	If $X$ has an even order, then $X$ contains a unique involution.
\end{lemma}

The proof will also use the following remark.

\begin{remark}\label{centr}
	Let $x$ be an element of $G$ of prime order $p.$ If $C_G(x)$ contains three cyclic subgroups $\langle a_1\rangle, \langle a_2\rangle, \langle a_3\rangle$ of prime orders, say $p_1,p_2,p_3,$ and $p\notin \{p_1,p_2,p_3\}$, then $(a_1x,a_2x,a_3x\to x)$ is a claw in $P^*(G).$
\end{remark}

\begin{lemma}\label{pabq}
	Let $G=P\rtimes Q$, where $P$ is a non-cyclic elementary abelian $p$-group and $Q$ is a $q$-group, with $p$ and $q$ different primes. If $P^*(G)$ is claw-free, then $C_P(y)= 1$ for every $1\neq y\in Q.$
\end{lemma}
\begin{proof}
Assume that, for some $1\neq y\in Q$, $C:=C_{P}(y)\neq 1.$ If $C$ is not cyclic, then it contains at least $p+1$ subgroups of order $p$, but then, by Remark \ref{centr}, $P^*(G)$ contains a claw. So $C=\langle x\rangle$, with $|x|=p.$ Moreover $1=[x^z,y^z]=[x,y^z]$ for every $z\in P$: this implies that $|P:C|\leq 2,$ otherwise we have at least $3$ different conjugates of $\langle y\rangle,$ all commuting with $x.$ 
If $p\neq 2,$ we would have $C=P=\langle x\rangle,$ 
against the assumption that $P$ is not cyclic. We reach the same conclusion if $p=2.$ Indeed if $C\neq P$, then $P\cong C_2\times C_2$ and $y$ centralizes a subgroup of $C$ of order 2. 
This would imply that $y$ centralizes $P$ and $C=P,$ a contradiction. 
\end{proof}

\begin{proof}[Proof of Theorem \ref{fourprime}]
	Suppose that $G$ is a minimal counterexample. Since $P^*(G)$ is claw-free, $P^*(H)$ is claw-free for every subgroup $H$ of $G$. In particular $\pi(H)\leq 4$ for every proper subgroup $H$ of $G$. Let $N$ be a maximal normal subgroup of $G$. Then $|G:N|=p_1$ is a prime and $p_1$ does not divide $N$, otherwise $\pi(G)=\pi(N)\leq 4.$ In particular
	$G=N\rtimes C_{p_1}$ and $\pi(N)=4.$ Now let $N/M$ be a chief factor of $G.$ 
	Then $N/M$ is a $p_2$-group for some prime $p_2.$ Moreover, if $P_1$ is a Sylow $p_1$-subgroup of $G$, then $MP_1<G$ and therefore $\pi(MP_1)\leq 4.$ Since $\pi(G)=\pi(MP_1)\cup \{p_2\}$, 
	we must have that $p_2\notin \pi(MP_1)$ and therefore $G=(M\rtimes P_2)\rtimes P_1,$ being $P_2$ a Sylow $p_2$-subgroup of $N.$ Repeating the argument, we may conclude that $G=(((P_5\rtimes P_4)\rtimes P_3)\rtimes P_2)\rtimes P_1$ where $P_i$ is a Sylow $p_i$-subgroup of $G$ and its is elementary abelian; more precisely $P_i$ is isomorphic to a chief factor of $G.$
	
	We distinguish two cases:
	
\noindent	1) There exists $1\neq x \in P_5$ such that $C_G(x)>P_5$.
	Then there exists $y$ of order $q$ (with $q\in \{p_1,p_2,p_3,p_4\}$) such that $[x,y]=1.$ By Lemma \ref{pabq}, $P_5=\langle x\rangle.$ So $X:=C_G(x)=C_G(P_5)$ is a normal subgroup of $G$. By Remark  \ref{centr}, $X$ must contain a unique subgroup of order $q$, so $\langle y \rangle$ is characteristic in $X$ (and therefore normal in $G$). But then be may assume  $P_5P_4=\langle xy \rangle.$ Now consider $Y=C_G(xy)$: it must be $Y=\langle x,y\rangle$ (otherwise we would have in $G$ an element of order divisible by 3 different primes). This $G/C_G(xy)\cong P_3P_2P_1$ is abelian, a contradiction.
	
\noindent	2) $H=P_4P_3P_2P_1$ acts fixed point freely on $P_5$. 
	But then $H$ is a Frobenius complement, and we may apply Lemma \ref{frocomp}. If $|H|$ is even, then $H$ contains a unique involution and this involution commutes with all the elements of $H$ of prime order In this case we may assume $p_4=2$. If $|H|$ is odd, then
	every two elements of $H$ of prime order commute. In any case, $G$ would contain cyclic subgroups
	of order $p_4p_3,$ $p_4p_2,$ $p_4p_1$, containing the same subgroup of order $p_4.$
\end{proof}

In order to see that the bound $\pi(G)\leq 4$ in Theorem \ref{fourprime} is best possible, we analyze what happens if $|G|=p_1p_2p_3p_4$, with $p_1, p_2, p_3, p_4$ four distinct primes. 

\begin{proposition}\label{squarefree}Suppose that the order of $G$ is the product of four distinct primes.
	Then $P^*(G)$ is claw-free if and only if $G\cong C_{p_1p_2}\rtimes C_{p_3p_4}$, where $p_1, p_2, p_3, p_4$ are distinc primes, and the centralizer in $G$ of every non-trivial element of $C_{p_1p_2}$ coincides with
	$C_{p_1p_2}.$ 
\end{proposition}

\begin{proof} Since all the Sylow subgroups of $G$ are cyclic, $G$ is solvable (see \cite[Corollary 5.15]{isa}). Let $F$ be the Fitting subgroup of $G$ and assume that $P^*(G)$ is claw-free. Then $F$ is cyclic, so, since $P^*(F)$ is claw-free, we must have $\pi(F)\leq 2.$ Thus we may assume that either $|F|=p_1$ or $|F|=p_1p_2.$ 
	
	First assume  $|F|=p_1.$ By \cite[5.4.4]{Robinson}, $C_G(F)=F$, so $G/F$ is isomorphic to a subgroup of $\aut(F)$. But $\aut(C_{p_1})$ is cyclic, so $G/F$ is cyclic of order $p_2p_3p_4$. Hence $G$ contains an element of order $p_2p_3p_4$, in contradiction with Remark \ref{manyprimes} (1).
	
	Assume now that $|F|=\langle x \rangle$ is cyclic of order $p_1p_2.$ Again we use the fact that $C_G(F)=F$ to deduce that $G/F$ is an abelian group. This implies that $G=F\rtimes K$, with $K\cong C_{p_3}C_{p_4}.$ Now let $1\neq y\in F.$ By Proposition \ref{centralizer},
	$\pi(C_G(y))\leq 2$. Since $F\leq C_G(y),$ we deduce $F=C_G(y).$

	Conversely assume $C_G(P_1)=C_G(P_2)=F$.
	Then the nontrivial subgroups of $G$ are precisely the following: $F, P_1, P_2,$  $p_1p_2$ subgroups of order $p_3$, $p_1p_2$ subgroups of order $p_4$ and $p_1p_2$ subgroups of order $p_3p_4.$
	In particular every subgroup of prime order in $G$ is property contained in a unique cyclic subgroup of $G$ and therefore  $P^*(G)$ is claw-free. 
\end{proof}

\section{Non solvable groups}\label{unso}


Before to prove Theorem \ref{unsolva} we need some  preliminary  results. The following is an immediate consequence of Proposition \ref{centralizer}.

\begin{proposition}\label{nocenter}
	Let $G$ be a finite group. If $P^*(G)$ is claw-free and $\pi(G)\geq 3$, then $Z(G)= 1.$ In particular if $P^*(G)$ is claw-free and $G$ is not solvable, then $Z(G)=1.$
\end{proposition}

\begin{proposition}
	\label{genfit}
	If $P^*(G)$ is claw-free, then either $G$ is almost simple or $F(G)=F^*(G)$, where $F(G)$ and $F^*(G)$ are the Fitting subgroup and the generalized Fitting subgroup of $G$ respectively.
	\begin{proof}
		Suppose $F^*(G)\neq F(G).$ Then $G$ contains a subnormal quasisimple subgroup $X.$ By Proposition \ref{nocenter} $Z(X)=1$ and therefore  $X$ is a non-abelian simple group. Moreover $F^*(X)=XC,$ with $C=C_{F^*(G)}(X).$ It follows from Theorem \ref{inde} that $C=1.$ 
		So we conclude that $X=F^*(G)$ is a non-abelian simple group and therefore $G$ is almost simple.
	\end{proof}
\end{proposition}

\begin{proof}[Proof of Theorem \ref{unsolva}]
	Let $G$ be a minimal counterexample. In particular $G$ is not almost simple, so by Proposition \ref{genfit}  $F^*(G)=F(G).$ We will use several times in the proof that
	this implies $F(G)\neq 1$ and $C_G(F(G)) = Z(F(G)).$

Let $R(G)$ be the solvable radical of $G$ and $Q$ a Sylow subgroup of $R(G).$ By the Frattini argument, $G=R(G)N_G(Q).$ Assume that $N_G(Q)<G.$ Then $N_G(Q)$ is not solvable, and $P^*(N_G(Q))$ is claw-free, hence by the minimality of $G$, it follows that $N_G(Q)$ is almost simple. But this is not possible, since $Q$ is normal in $N_G(Q).$
	 Thus all the Sylow subgroups of $R(G)$ are normal in $G$. But then $R(G)$ is nilpotent and therefore $F(G)=R(G).$ 
		If $F(G)$ is not a $p$-group, then by Proposition \ref{th_claw_nil}, $F(G)$ is cyclic. Since $C_G(F(G))\leq F(G),$ we would have that $G/F(G)\leq \aut(F(G))$ is abelian. Thus $G$ is solvable, a contradiction. So $F(G)$ is a $p$-group, for a suitable prime $p.$ 
		
Let $N/F(G)$ be a minimal normal subgroup of
$G/F(G).$ Since $F(G)=R(G),$  $N/F(G)$ is non-solvable. If $N\neq G,$ then, by the minimality of $G$, $N$ is almost simple, in contradiction with the fact that $F(G)$ is a non-trivial nilpotent normal subgroup of $N$. Thus $G=N$ and $G/F(G)$ is isomorphic to a non-abelian simple group $S.$

	We claim that $p=2.$ Indeed if $p\neq 2,$ then a Sylow 2-subgroup $P$ of $G$ is isomorphic to a Sylow 2-subgroup of $S.$ In particular $P$ contains a subgroup isomorphic to $C_2\times C_2.$ By Lemma \ref{frocomp}, the action of $P$ on $F(G)$ is not fixed-point-free, hence $F(G)P$ is not claw-free since it contains a subgroup isomorphic to
	$C_p\times C_2\times C_2.$
	
Suppose that $F(G)$ contains an element $g$ of order 4. Then
	the subgroup $X$ generated by the elements of order 2 in $Z(F(G))$ has order at most 4 (otherwise $C_F(g)$ would contain a subgroup isomorphic to
	$C_4\times C_2\times C_2))$. If follows from Proposition \ref{nocenter}, that $C_G(X)=F$. But then $G/F$ is isomorphic to a subgroup of $\aut(X),$ which is solvable, a contradiction. So $F$ is an elementary abelian 2-group.
		
Let $P$ be a Sylow 2-subgroup of $G$. Notice that $P$ is non-abelian, otherwise $P\leq C_G(F(G))=F(G),$ in contradiction with the fact that $P/F(G)$ is isomorphic to a Sylow 2-subgroup of a non-abelian simple group. Moreover $P$ cannot be a dihedral group; otherwise $F$ would be a normal subgroup of index at least 4 of a dihedral group, and therefore $F$ would be cyclic and $G/F\leq \aut(F)$ would be abelian. It follows from Proposition \ref{2groups} that $P$ has exponent 4, so we may apply Proposition \ref{quatfree} to deduce that $P=N\langle x\rangle,$ with $N$ an abelian normal subgroup of $P.$ Morevor, since $P/F(G)$ is isomorphic to a Sylow 2-subgroup of a non-abelian simple group, $P/F(G)$ is not cyclic. 

First assume that $N\leq C_4\times C_4.$ This implies $F(G)\leq C_2\times C_2 \times C_2.$ Since $G/F(G)$ is a non-abelian simple group, the unique possibily is that $F(G)\cong C_2\times C_2\times C_2$ and $G/F(G)\cong {\rm{SL}}(3,2).$ In particular $G$ is perfect of order 1344. Using the library of perfect groups available in GAP \cite{GAP}, it can be checked  that there are two possibilities. Either $G\cong \rm{{AGL}}(3,2)$ or $F(G)$ is the Frattini subgroup of $G.$ In the first case in $G$ there are 16 subgroups of order 6 containing the same involution, in the second case $G$ contains an element of order 8. So in both the cases, $P^*(G)$ is not claw-free.

Now assume that $N$ is an elementary abelian 2-group. 
Since $P/F(G)$ is not cyclic, $N$ is not contained in $F(G).$ This implies  $N\not\leq C_P(F(G)).$ 
Let $y\in F(G)\setminus C_P(N).$ Then $\langle y\rangle N$ is non-abelian and therefore there exists $n\in N$ such that $z=yn$ has order 4. Since $P$ contains no subgroup isomorphic to $C_4\times C_2\times C_2,$ we deduce $|C_N(z)|\leq 4.$
Since $F(G)$ is abelian, $F(G)\cap N\leq C_N(z)=C_N(y)\leq C_2\times C_2$. This would imply $F(G)\leq C_2\times C_2 \times C_2$. We can conclude as in the previous paragraph.
\end{proof}


\section{Non-abelian simple groups}\label{simgr}

It is well known that a finite non-abelian simple group contains a four-subgroup, since no simple group has a cyclic group or a quaternion group for its Sylow 2-subgroup. Moreover
if $G$ is a finite non-abelian simple group and $P^*(G)$ is claw-free then no subgroup of $G$ can be isomorphic to $C_2\times C_2\times C_p$ with $p$ an odd prime. This implies that the centralizer of every four-subgroup is a $2$-group. Thus a key role in the study of finite simple groups whose 
reduced-power graph is claw-free is played by the following result:
\begin{theorem}\cite[Theorem 1]{sys} A finite simple group in which the centralizer of every four-subgroup is a $2$-group is isomorphic with one of the following: \begin{enumerate}\item $\psl(2,q),$
	${\rm{Sz}}(q),$ ${\rm{J}}_1,$ ${\rm{M}}_{11},$ $\psl(3,4),$ $^2{\rm{F}}_4(2)^\prime,$ \item $\psl(3,q),$ $\psu(3,q)$ or $\g_2(q)$, for suitable odd $q\geq 3.$
	\end{enumerate}
\end{theorem}

In order to prove Theorem \ref{simplenoclaw} it suffices to investigate the groups listed in the statement of the previous theorem.
\begin{lemma}
$P^*(J_1)$ is not claw-free.
\end{lemma}
\begin{proof}
It follows from Lemma \ref{nocenter}, noticing that $J_1$ contains a subgroup isomorphic to $C_2\times A_5$ \cite{Conway}.
\end{proof}

\begin{lemma}
	$P^*({\rm{M}}_{11})$ is not claw-free.
\end{lemma}

\begin{proof}
	It suffice to notice that any involution of ${\rm{M}}_{11}$ is contained in three different cyclic subgroups of order 4.
\end{proof}

\begin{lemma}
	Let $G={\rm{Sz}}(q)$, where $q=2^{2k+1}$, where $k\geq 1$ is a positive integer. Then  $P^*(G)$ is never claw-free.
\end{lemma}

\begin{proof}
	According to \cite{suzuki2}, a Sylow $2$-subgroup of $G$ has order $q^{2}$; it contains $q-1$ elements of order $2$, and $q^{2}-q$ elements of order $4$. So, at least $\dfrac{q}{2}$ subgroups of order $4$ share a common subgroup of order $2$. Since $q/2\geq 3,$ this implies that the reduced power graph of a Sylow $2$-subgroup of $G$ always contains a claw.
\end{proof}

\begin{lemma}
	$P^*(\psl(3,4))$ is not claw-free.
\end{lemma}
\begin{proof}
It suffices to prove that in $\psl(3,4)$ every involution is the square of precisely 12 elements of order 4.
\end{proof}

\begin{lemma}\label{psl3}
If $q$ is odd, then $P^*(\psl(3,q))$ is not claw-free.
\end{lemma}

\begin{proof}$ \psl(3,q)$ contains a subgroup $H\cong {\rm{SL}}(2,q).$ If $q\neq 3$, then $H$ is not solvable and $Z(H)$ has order 2. But then $P^*(H)$ contains a claw by Proposition \ref{nocenter}. It is not difficult to check that
	$\psl(3,3)$ contains 3 different cyclic subgroups of order 4 containing the same subgroup of order 2.
\end{proof}

\begin{lemma}
If $q$ is odd, then 	$P^*(\g_{2}(q))$  is not claw-free.
\end{lemma}
\begin{proof}
When $q$ is odd,	$\g_{2}(q)$ contains  a subgroup $H$ isomorphic to ${\rm{SL}}(3, q)$ \cite{kleidman2}. If $Z(H)=1,$ then
	$H\cong \psl(3,q)$ and we conclude from the previous lemma that $P^*(H)$ contains a claw. It $Z(H)\neq 1,$ then the same conclusion is ensured by Proposition \ref{nocenter}.
\end{proof}

\begin{lemma}
	$P^*(^2{\rm{F}}_4(2)^\prime)$ is not claw-free.
\end{lemma} 
\begin{proof}
It suffices to notice that $^2{\rm{F}}_4(2)^\prime$ contains a subgroup $H\cong \psl(3,3)$ \cite{Conway} and that $P^*(H)$ is claw-free by Lemma 	\ref{psl3}.
\end{proof}

\begin{lemma}
	If $q$ is odd, then $P^*(\psu(3, q))$ is not claw-free.
\end{lemma}

	\begin{proof}$ \psu(3,q)$ contains a subgroup $H\cong {\rm{SL}}(2,q).$ If $q\neq 3$, then $H$ is not solvable and $Z(H)$ has order 2. But then $P^*(H)$ contains a claw by Proposition \ref{nocenter}. If $q=3$ the conclusion follows from the fact that every element of $\psu(3,3)$ of order 2 is contained in 4 different cyclic subgroups of order 4.
	\end{proof}

\begin{theorem}\label{psl2}
	Let $G=
\psl(2, q)$ and let $d=(q-1,2).$
Then $P^*(G)$ is claw-free if and only if  each of the two  numbers $(q \pm 1)/d$ is either a prime power or the product of some prime with a prime-power.
\end{theorem}

\begin{proof}
Let $\Omega$ be the set of positive integers that are either prime-power or of the form $p^nq,$ where $p$ and $q$ are distinct primes and $n\in \mathbb N.$ Moreover let
$z_1=(q-1)/d$ and $z_2=(q+1)/d.$ Assume that $q=p^t,$ with $p$ a prime, and let $K$ be a Sylow $p$-subgroup of $G.$ By \cite[II, Satz 8.5]{hu}
$G$ contains a cyclic subgroup $H_1$ of order $z_1$ and a cyclic subgroup $H_2$ of order $z_2$ and the conjugates of $K, H_1, H_2$ give rise to a partition of $G.$ Thus $P^*(G)$ is claw-free if and only $P^*(K), P^*(H_1)$ and $P^*(H_2)$ are claw-free. On the other hand $P^*(K)$ is claw-free since $K$ is an elementary abelian $p$-group, while, by Remark
 \ref{manyprimes}, for $i\in \{1,2\},$ $P^*(H_i)$ is claw-free if and only if $z_i\in \Omega.$
\end{proof}

\section{Proof of Theorem \ref{fiveprime}}

\begin{proof}[Proof of Theorem \ref{fiveprime}]
Assume that $P^*(G)$ is claw-free. If $G$ is solvable, then $\pi(G)\leq 4$ by Theorem \ref{fourprime}. So assume that $G$ is almost simple. Let $S=\soc(G).$ By Theorem \ref{simplenoclaw}, $S=\psl(2,q),$ with $q=p^f$ for a suitable prime $p.$ Let $d=(q-1,2).$ Since $|S|=q(q-1)(q+1)/d,$ it follows from Theorem \ref{simplenoclaw}
that $\pi(S)\leq 5.$ Recall that the order of $G/S$ divides $2f.$
Assume that $r$ is a prime dividing $|G|$ but not $|S|$.
Then $r$ divides $f$. Let $u=f/r$ and let $g$ be an element of order $r$ in $f$. Then $C_S(g)\cong \psl(2,p^u)$ and $G$ contains a subgroup $H$ isomorphic to $\psl(2,p^u)\times C_r.$
However $|\psl(2,p^u)|$ is divisible by at least two different primes and is coprime with $r$, so $P^*(H)$ contains a claw by Theorem \ref{inde}.
\end{proof}

\section{Acknowledgements}
Pallabi Manna expresses gratitude to the Department of Atomic Energy (DAE), India for providing funding (in the form of post-doctoral fellowship) during this work. Santanu Mandal acknowledges the infrastructure provided by VIT Bhopal University, India. 
\section{Statements and Declarations} \textbf{Conflict of interest:} On behalf of all authors, the corresponding author states that there
is no conflict of interest.
\section{Data Availability}
Data sharing is not applicable to this article as no data were created or analyzed in this study.

\end{document}